\documentclass{article}
\usepackage{amsmath, amsthm, amssymb}
\usepackage{engord}

\newcounter{statements}
\numberwithin{statements}{section}
\newtheorem{thm}[statements]{Theorem}
\newtheorem{prop}[statements]{Proposition}
\newtheorem{lem}[statements]{Lemma}

\newtheorem{cor}[statements]{Corollary}
\newtheorem{defn}[statements]{Definition}
\newtheorem{obs}[statements]{Observation}

\newtheorem{prob}[statements]{Problem}
\newtheorem{rem}[statements]{Remark}
\newtheorem{example}[statements]{Example}

\DeclareMathOperator{\NA}{NA}
\DeclareMathOperator{\BI}{BI}
\DeclareMathOperator{\Diff}{Diff}

\DeclareMathOperator{\graph}{graph}

\newcommand{\Wls}{W_{loc}^s}
\newcommand{\Wlu}{W_{loc}^u}

\newcommand{\Wsu}{W_{gl}^u}
\newcommand{\DS}{\mathcal{DS}}
\newcommand{\M}{\mathcal{M}}

\title{Stable ergodicity of dominated systems}
\author{Martin Andersson \thanks{This work was carried out at Universidade de S\~{a}o Paulo, S\~{a}o Carlos, and \'Ecole Normale Sup\'erieure, Paris, with financial support from FAPESP (Brazil) and Fondation Sciences Math\`ematiques \`a Paris (France). }\\
 }

\begin{document}
\maketitle

\abstract{ We provide a new approach to stable ergodicity of systems with dominated splittings, based on a geometrical analysis of global stable and unstable manifolds of hyperbolic points. Our method suggests that the lack of uniform size of Pesin's local stable and unstable manifolds --- a notorious problem in the theory of non-uniform hyperbolicity --- is often less severe than it appeas to be.}

\section{Introduction}

The theory of ergodicity of diffeomorphisms deals with the
question of whether a given  conservative (volume preserving)
diffeomorphism is ergodic. It dates back to the work of Anosov
\cite{MR0163258}, who adopted Hopf's argument \cite{MR0001464} on the ergodicity
of certain geodesic flows, to prove that every conservative $C^2$
Anosov diffeomorphism on a compact connected Riemannian manifold is ergodic.
Since the set of Anosov diffeomorphisms is open in the $C^1$ topology, such
diffeomorphisms are stably ergodic in the following sense: every
conservative $C^2$ diffeomorphism sufficiently close, in the $C^1$
topology, to a conservative $C^2$ Anosov diffeomorphism is
ergodic. This is the definition of stable ergodicity used throughout this paper.

In \cite{MR1449765} Pugh and Shub initiated a programme for studying
ergodicity of partially hyperbolic diffeomorphisms. They conjectured that stable ergodicity, as defined above, is $C^2$-dense in this context. The conjecture has been very successful and recently proved  in \cite{MR2390288} to be true in the case where the central bundle is one dimensional. 

Although the theory of stable ergodicity so far has dealt
predominantly with systems admitting a partially hyperbolic
splitting, such is by no means necessary. In his thesis \cite{MR2085722}, Tahzibi provided examples of stably ergodic diffeomorphisms with no
uniformly expanding/contracting subbundle. On the other hand, his
examples enjoy the following properties:
\begin{itemize}
\item They admit an invariant  dominated splitting $TM = E^{cs} \oplus E^{cu}$ (see definition in section \ref{statements}).
\item They are non-uniformly hyperbolic,  meaning that all Lyapunov exponents
are non-vanishing in a full Lebesgue measure set.
\item The Oseledet splitting is compatible with the dominated
splitting in the sense that, for Lebesgue almost every $x \in M$,
the Lyapunov exponents of vectors in $E_x^{cs}$ are all negative, and those of vectors in $E_x^{cu}$ are all positive.
\end{itemize}
Conservative diffeomorphisms with these properties will henceforth
be referred to as \emph{non-uniformly Anosov diffeomorphisms}. They are the subject of
this article. Indeed, following the apparent success of the
Pugh-Shub conjecture on the denseness of stable ergodicity in
partially hyperbolic dynamics, it seems natural to conjecture
denseness of stable ergodicity among conservative diffeomorphisms
admitting a dominated splitting. This paper is written along with the belief that, within the space $\mathcal{DS}$, of $C^2$ diffeomorphisms admitting a dominated splitting, there is an open and dense set of ergodic, non-uniformly Anosov diffeomorphisms. Here it is relevant to mention the result of Bochi,
Fayad, Pujals \cite{MR2227756}, in which it is proved that every
stably ergodic diffeomorphism is $C^1$ approximated by one which
is non-uniformly Anosov.

Thus motivated by \cite{MR2085722} and \cite{MR2227756}, we turn the question around and ask: under what condition are non-uniformly Anosov diffeomorphisms (stably) ergodic? We can try, naively, to apply the Hopf-Anosov argument. On doing so, one has to take into account two
differences between the Anosov and non-uniform Anosov case:
\begin{enumerate}
\item Non-uniformly Anosov diffeomorphisms do not have local stable
and unstable manifolds associated to every point, but only to
Lebesgue almost every point.
\item In the non-uniformly Anosov case, local stable and unstable
manifolds are not of uniform size.
\end{enumerate}

Only the lack of uniform size of local stable and unstable
manifolds provides a serious problem for the adaptation of the
Hopf-Anosov argument, and actually makes it fail in some cases (not
all non-uniformly Anosov diffeomorphisms are ergodic, see \cite{DHP}). Still, the problem is not necessarily as serious as it may first seem, by the following 
\begin{obs}
Rather than the local stable and unstable manifolds, what matters in the Hopf-Anosov argument is the size (whatever that means) of the stable and unstable sets
\begin{align}
 W^s(x) &= \{y \in M: d(f^n(x), f^n(y)) \rightarrow 0 \text{ as } n \rightarrow \infty \} \\
W^u(x) &= \{y \in M: d(f^{-n}(x), f^{-n}(y)) \rightarrow 0 \text{ as } n \rightarrow \infty \}.
\end{align}
\end{obs}
More precisely, by local stable manifolds we mean
the sets
\begin{multline}
\Wls (f, x):= \{ y\in M : d(f^n(x), f^n(y)) \leq r_0 \ \forall n \geq 0
\\ \text{ and } \limsup_{n \rightarrow \infty} \frac{1}{n} \log d(f^n(x), f^n(y)) < 0 \},
\end{multline}
$r_0$ being a small constant. $\Wlu(f, x)$ is defined analogously. It is a well-known theorem of Pesin \cite{MR0458490} that, for $m$-almost every $x \in M$, $\Wls (f, x)$ is a $C^2$ embedded disk, tangent to $E^{cs}$ in our setting. 
The importance of the stable sets is that if, for some $x \in M$, continuous $\varphi: M \rightarrow \mathbb{R}$,  the forward Birkhoff average
\begin{equation*}
 \varphi_+ (x) := \lim_{n \rightarrow \infty} \frac{1}{n} \sum_{k=0}^{n-1} \varphi( f^k(x))
\end{equation*}
exists, then it also exists, and coincides, for every $y \in W^s(x)$, i.e. $\varphi_+(y) = \varphi_+(x)$. Similarly for $W^u(x)$ and averages in backword time.
However, $W^s(x)$ and $W^u(x)$ are sets
about which we know very little in general. In particular, it is not clear whether there is a natural way to talk about the size of $W^s(f,x)$. 
It is therefore more convenient to work with the \emph{global stable manifold} of $x$: 
\begin{equation*}
\Wsu(f, x):= \bigcup_{n\geq 0} f^n (\Wlu(f^{-n}(x))) \subset W^u(f,x).
\end{equation*}
It is an immersed, rather than embedded manifold, and it is known from Pesin's work that $\Wsu(x)$ can be characterised as the set of those $y\in M$ for which $d(f^n(x),f^n(y))$ converge to zero exponentially fast. Everything we define, state, or prove about local or global unstable manifolds has a counterpart for local or global stable manifolds and vice verse. The necessary notation and terminology should be obvious and we will not always bother to mention it. 

Even though the rate of convergence of $d(f^n(x),f^n(y))$ is irrelevant to the Hopf-Anosov argument, we shall concern ourself with the structure of $\Wsu(f,x)$ rather than $W^u(f,x)$. Being a countable union of nested embedded disks, the former is simply much more tangible. We are going to  define a notions of size of $\Wlu(f,x)$, called the \emph{span} and the \emph{essential span}. For the purpose of this introduction, it suffices to say that $W^s(f,x)$ has span (at least) $\delta>0$ if it contains a round disk (intrinsic ball) of radius $\delta$, centered at $x$. In the case of essential span, we allow the disk to be perforated by little holes of zero measure (see Definition \ref{span}).  Thus, in order to say something general about the size of $\Wsu(x)$ for Lebesgue almost every $x$, we should find a way to approach the following question:

\begin{prob}
What condition guarantees that most (Lebesgue almost all) local unstable (stable) manifolds grow under forward (backward) iteration of the diffeomorphism?
\end{prob}

Some clarification is called for. It is not difficult to show that, given any non-uniformly Anosov diffeomorphism $f$, and Lebesgue almost any point $x \in M$, the volume of $f^n(\Wlu(x))$ grows exponentially fast as $n \rightarrow \infty$. The problem is \emph{how} it grows. If the growth resembles that of an inflating balloon, one finds that, for large values of $n$, the larger part (in terms of volume) of points in $f^n(\Wlu(x))$ admit a ball of some fixed radius $\delta>0$, entirely contained in $f^n(\Wlu(x))$. We prove that this is sufficient to guarantee Lebesgue almost everywhere uniform span of global unstable manifolds. If, on the other hand, the growth of $f^n(\Wlu(x))$ resembles that of a growing worm or --- even worse --- a growing tree, the same cannot be said. Hence the first type of growth is the desired one when proving ergodicity. Motivated by this balloon vs. tree thinking, we define the following notion: $f$ is said to be $cu$-inflatable if 
\begin{equation}\label{condition}
 \int \log \vert \det Df_{\vert E_x^{cu}} \vert dm
 > \log \left( \sup_{x\in M} \| \wedge^{\dim E^{cu}-1}Df_{\vert E_x^{cu}} \| \right).
\end{equation}

Here $\| \wedge^{\dim E^{cu}-1}Df_{\vert E_x^{cu}} \| $ denotes the maximum expansion rate, under the action of $Df$, of the volume of a $cu-1$-dimensional parallelpiped contained in $E_x^{cu}$.
Roughly, the $cu$-inflatability condition says that the average expansion, of a $cu$-dimensional volume element in $E^{cu}$, is larger than the maximum expansion of a $(cu-1)$-dimensional volume element in $E^{cu}$. However, the presence of logarithms in inequality (\ref{condition}) means that we ask for slightly more than that.
The notion of $cs$-inflatability is defined analogously. If a non-uniform Anosov diffeomorphism is both $cs$- and $cu$-inflatable, we say that it is bi-inflatable. 

\subsection*{We prove:}

\begin{itemize}
 \item $cu$-inflatability of a non-uniformly Anosov 
diffeomorphism implies the existence of at least one ergodic component on which Lebesgue almost every point has a global unstable manifold of 
infinite (arbitrarily large) essential span. 

\item Similarly, $cs$-inflatability implies the existence of at least one ergodic component on which Lebesgue almost every point has a global stable manifolds of 
infinite essential span. 
\item If some ergodic component has, Lebesgue almost everywhere, both global stable and global unstable manifolds of uniform essential span, then the component has full Lebesgue measure, so that the diffeomorphism is ergodic.
\item This situation persists under small $C^1$ perturbations, so the diffeomorphism is, in fact, stably ergodic.
\item Transitivity of a bi-inflatable non-uniformly Anosov diffeomorphism,  with $E^{cu}$ plaque uniquely integrable, implies stable ergodicity.
\item On tori, stable ergodicity can also be guaranteed for bi-inflatable non-uniformly Anosov diffeomorphisms by assuming the subbundles $E^{cu}$ and $E^{cs}$ to be approximately constant.
\item If $f$ is a non-uniformly Anosov diffeomorphism for which one of the $E^{\sigma}$, $\sigma = cs, cu$, is uniformly contracting/expanding, and the other one is inflatable, then $f$ is stably ergodic.
\end{itemize}

\subsection*{Acknowledgements}

The main ideas in this work were developed during my stay as a postdoc at ICMC-USP, S\~ao Carlos, under the supervision of Ali Tahzibi, to whom I owe much gratitude. The problem of the size of global stable manifolds germinated form our frequent discussions and owes much to him. The idea that transitivity may be used in order to obtain ergodicity of inflatable systems came from Alexander Arbieto in the form of spontaneous remarks during an early exposition of the work. Thank you. I would also like to thank Paul Schweitzer for being a patient listener and good adviser on geometric matters, and Jairo Bochi for showing interest and encouragement.

\section{Precise statement of results}\label{statements}

Throughout this paper, $M$ denotes a compact connected Riemannian manifold. Its dimension is required to be at least two, but interesting examples only appear in dimension greater than or equal to $3$.
The volume form obtained from the Riemannian metric induces a measure on the Borel $\sigma$-algebra. We denote its normalisation by $m$ and refer to it as Lebesgue measure. As usual, $\Diff_m^2(M)$ stands for the space of all volume preserving $C^2$-diffeomorphisms on $M$, endowed with the $C^2$ topology. A diffeomorphisms $f  \in \Diff_m^2(M)$ is said to be ergodic if $m$ is an ergodic measure for $f$. Furthermore, $f\in \Diff_m^2(M)$ is said to be stably ergodic if there exists a $C^1$-neighbourhood $\mathcal{U}$ of $f$ such that every $g \in \mathcal{U} \cap \Diff_m^2(M)$ is ergodic.

\begin{defn}
 We say that $f \in \Diff_m^2(M)$ admits a dominated splitting if we can write 
$T M = E^{cs} \oplus E^{cu}$, where $E^{cs}$ and $E^{cu}$ are non-trivial complementary subbundles, invariant under the action of $Df$:
\begin{align}
 Df(x) E_x^{cs} &= E_f(x)^{cs} \quad \forall x \in M, \\
 Df(x) E_x^{cs} &= E_f(x)^{cs} \quad \forall x \in M,
\end{align}
and if there exist numbers $C>0$, $1>\tau>0$, such that, at every $x \in M$, 
\begin{equation}\label{domination}
 \| Df_{\vert E_x^{cs}}^n \| \cdot \| (Df_{\vert E_x^{cu}}^n)^{-1} \| \leq C \tau^n \quad \forall n \geq 0.
\end{equation}

\end{defn}
Let $\DS \subset \Diff_m^2(M)$ denote the space of all $C^2$ volume preserving diffeomorphisms on $M$, admitting a dominated splitting $TM = E^{cs} \oplus E^{cu}$.
It is a well-known consequence of the uniform domination property (\ref{domination}) that $x \mapsto E_x^{cs}$ and $x \mapsto E_x^{cu}$ are continuous. Hence, by compactness of $M$, there is a uniform (in $x$) lower bound on the angle between $E_x^{cs}$ and $E_x^{cu}$. The labels $cs$ and $cu$ stand for `central stable' and `central unstable' (bundle), and are also used to denote the dimensions of these bundles.

 For $f \in \DS$, there is often more than one choice of dominated splitting, but it becomes unique once we fix the dimension of (say) $E^{cs}$. Throughout this paper we shall therefore treat $cs$ as a fixed integer $1\leq cs \leq d-1$ so that there is no harm in talking about \emph{the} dominated splitting $E^{cs}\oplus E^{cu} = E^{cs}(f)\oplus E^{cu}(f)$ of $f\in \DS$. By characterising dominated splittings in terms of invariant conefields, one can see that $E^{cs}(f)$ and $E^{cu}(f)$ persist under small $C^1$-perturbations of $f$. That is, if $g\in \Diff_m^2(M)$ is close enough to $f$ in the $C^1$ topology, then $g$ has a dominated splitting $TM = E^{cs}(g) \oplus E^{cu} (g)$ with $\dim E^{\sigma} (g)= \dim E^{\sigma}(f)$ for $\sigma = cs, cu$, and with the $E^{\sigma}(g)$ close to the $E^{\sigma}(f)$ in corresponding Grassmannian bundles. In particular, $\DS$ is open, not only in the $C^2$-, but also in the $C^1$-topology.

\begin{defn}
We say that $f \in \Diff_m^2(M)$ is \emph{non-uniformly Anosov},
and write $f \in \NA$, if $f$ admits a dominated splitting
$TM = E^{cs} \oplus E^{cu}$ such that the asymptotic conditions
\begin{equation} \label{contraction}
  \lambda^{cs}(f,x) := \limsup_{n \rightarrow \infty} \frac{1}{n} \log \| Df_{\vert E_x^{cs}}^n \| < 0
  \end{equation}
  and
  \begin{equation}\label{expansion}
  \lambda^{cu}(f,x):= \limsup_{n \rightarrow \infty} \frac{1}{n} \log \|Df_{\vert E_x^{cu}}^{-n} \| < 0
  \end{equation}
 are satisfied at $m$-almost every $x \in M$.
\end{defn}

We state a version of Pesin's stable manifold theorem, adapted to our context. Given $f\in \NA$, we denote by $H = H(f)$ the set of hyperbolic points, i.e. those $x\in M$ for which (\ref{contraction}) and (\ref{expansion}) hold.

\begin{thm}[Pesin \cite{MR0458490}]
Given $f\in \NA$ and sufficiently small $r_0>0$, there exists a positive measurable function $r: H \rightarrow \mathbb{R}$, with 
\begin{equation*}
\lim _{n \rightarrow \pm \infty} \frac{1}{n} \log r(f^n(x)) = 0, 
\end{equation*}
and  such that, for every $x\in H$, 
\begin{align}
 \Wls(x)  := \{ y \in M: & d(f^n(x), f^n(y))< r_0 \text{ and } \\
&\lim_{n\rightarrow \infty} \frac{1}{n} \log d(f^n(x), f^n(y)) <0 \}, \\
 \Wlu(x)  := \{ y \in M: & d(f^{-n}(x), f^{-n}(y))< r_0 \text{ and } \\
&\lim_{n\rightarrow \infty} \frac{1}{n} \log d(f^{-n}(x), f^{-n}(y)) <0 \}
\end{align}
are $C^2$ embedded disks, tangent to $E^{cs}$ and $E^{cu}$ respectively, given by
\begin{equation*}
 \Wls(x)  := \exp_x (\graph(\psi_x^s)), \quad \Wlu(x) := \exp_x (\graph(\psi_x^u)),
\end{equation*}
 where $\psi_x^s: E_x^{cs}(r_0)\rightarrow E_x^{cu}$ and $\psi_x^u: E_x^{cu}(r_0) \rightarrow E_x^{cs}$ are $C^2$ maps with both $D\psi_x^s(0)$ and $D\psi_x^u(0)$ of zero rank.
\end{thm}

Given a $C^1$ immersed submanifold $N \subset M$ we can restrict the Riemannian metric of $M$ to $N$ to obtain a family of inner products $\{ \langle \cdot , \cdot \rangle_p^N : p \in N \}$, depending continuously on $p \in N$. With some abuse of language we call it a Riemannian metric, even though it is not (necessarily) smooth, and denote by $m_N$ the (not normalised) volume measure on $N$ inherited by it. We consider $m_N$ as a measure on $M$ by the obvious way by letting $m_N(E):= m_N(E\cap N)$ for every Borel set $E \subset M$. The total mas $m_N(M) = m_N(N)$ of $m_N$ is denoted by $\vert m_N \vert$. The intrinsic distance in $N$ is denoted by $d^N(\cdot, \cdot)$ and, for $\delta>0$, we write $B_{\delta}^N (x):= \{y\in N: d^N(x,y)<\delta \}$. Note that if $N$ is an immersed submanifold which is not actually a submanifold, the topology on $N$ induced by the metric $d^N(\cdot, \cdot)$ does not have to coincide with that induced by considering $N$ as a subset of $M$.

Let $D_1,$ and $D_2$ be $cs$-dimensional $C^1$-submanifolds transverse to $E^{cu}$, $\Lambda \subset H$,  and $\mathcal{W} = \{W_x := \Wlu(x): x \in \Lambda\}$ a family of local stable manifolds intersecting each $D_i, \ i=1,2$ in exactly one point. We may then define the holonomy map $h$, associated to $(\mathcal{W}, D_1, D_2)$ by
\begin{align}
 h: \{D_1 \cap W_x: x\in \Lambda\} & \rightarrow \{D_2 \cap W_x: x \in \Lambda\} \\
          D_1 \cap W_x & \mapsto D_2 \cap W_x.
\end{align}

\begin{thm}[Pesin \cite{MR0458490}]
 The holonomy map $h$ associated to $(\mathcal{W},D_1, D_2)$ maps sets of zero $m_{D_1}$-measure into sets of zero $m_{D_2}$-measure.
\end{thm}

Naturally there is an analogue for the holonomy map associated to local stable manifolds.

\begin{defn}\label{span}
We say that a $k$-dimensional $C^1$-immersed submanifold $N\subset M$, $k=1, \ldots d$, has $k$-span $\delta>0$ around $p\in N$ if for every $\epsilon>0$
\begin{equation} \label{annulus 1}
 B_{\delta}^{N}(x) \setminus \overline{B}_{\delta - \epsilon}^{N}(x) = 
\{ y \in N: \delta-\epsilon < d^N(x,y) < \delta \}
\end{equation}
is homeomorphic to the annulus
\begin{equation}\label{annulus 2}
 \{ z \in \mathbb{R}^{cu} : 1< \| z \| < 2 \}.
\end{equation}

A Borel subset $E \subset M$ has essential $k$-span $\delta>0$ around $x$ if there exists a $C^1$-immersed $k$-dimensional submanifold $N\subset M$ of such that $N$ has span $\delta$, $x\in N$,  and 
\begin{equation*}
 m_N(E\cap B_{\delta}^N (x))= m_N(B_{\delta}^N(x)).
\end{equation*}

If $\Wsu(x)$ has $cu$-span $\delta$ around $x$, we say that $x$ has an unstable manifold of span $\delta$.
If $\Wsu(x)$ has essential $cu$-span $\delta$ around $x$, we say that $x$ has an essential unstable manifold of span $\delta$.

If $\Wsu(x)$ has span (or essential span) $\delta$ for every $\delta>0$ we say that $\Wsu(x)$ has infinite span (or essential span).
\end{defn}

The same definitions carry over verbatim to the case of stable manifolds. Sometimes the term \emph{full} stable or unstable manifold will be used instead of stable or unstable manifolds to add more contrast to essential stable or unstable manifolds.

The requirement that (\ref{annulus 1}) be homeomorphic to (\ref{annulus 2}) is a way to make precise the statement that $\Wsu(x)$ contains a disk of radius $\delta$ around $x$. It can be explained as follows: Let $\Gamma = \{(x,y) \in \mathbb{R}^2: y>0 \} \subset \mathbb{R}^2$ be endowed with the usual Euclidian metric. Since $\Gamma$ is a subset of $\mathbb{R}^2$, 
it is meaningful to consider the set 
\begin{equation*}
 \Gamma_1 = \{ p \in \Gamma : B_1^{\mathbb{R}^2} (p) \subset \Gamma \}, 
\end{equation*}
i.e. $\Gamma_1 = \{(x, y) \in \mathbb{R}^2: y \geq 1 \}$, or the set of points with distance at least $1$ form $\partial \Gamma$. However, if $\Gamma$ is viewed intrinsically rather than as the subset of $\mathbb{R}^2$, then the definition of $\Gamma_1$ becomes meaningless. Nor does $\Gamma$ possess any boundary. However, by defining $\Gamma_1$ as the set 
\begin{equation*}
\{p \in \Gamma: B_{1}^{\Gamma}(p) \setminus \overline{B}_{1-\epsilon}^{\Gamma}(p) \text{ is homeomorphic to the annulus} \}, 
\end{equation*}
 one finds that $\Gamma_1$ is, again, equal to $\{(x, y) \in \mathbb{R}^2: y \geq 1 \}$.

Given a positive Lebesuge measure set $A \subset M$, we denote by $m_A$ the normalised restriction of Lebesgue measure to $A$. ( The reader should be aware that the notation $m_S$ means two different things depending on whether $S$ is a submanifold of $M$ or a Borel subset of positive $m$-measure.)

\begin{defn}
Let $f \in \mathcal{DS}$. An $f$-invariant Borel set $A$ of positive Lebesgue measure is $cu$-inflatable  if, for some $n \geq 1$,
\begin{equation*}
 \int \log \vert \det Df_{\vert E_x^{cu}}^n \vert dm_A
 > \log \left( \sup_{x \in M} \| \wedge^{cu-1}Df_{\vert E_x^{cu}}^n \| \right), 
\end{equation*}
Similarly, $A$ is $cs$-inflatable if, for some $n \geq 0$,
\begin{equation*}
 \int \log \vert \det Df_{\vert E_x^{cs}}^{-n} \vert dm_A
 > \log \left( \sup_{x \in M} \| \wedge^{cu-1}Df_{\vert E_x^{cs}}^{-n} \| \right).
\end{equation*}
If $A$ is both $cs$- and $cu$-inflatable we say that $A$ is bi-inflatable. If $M$ is $cs$-, $cu$-, or bi-inflatable, we attribute this property to the diffeomorphism itself by saying that $f$ is $cs$-, $cu$-, or bi-inflatable, respectively.
\end{defn}

\begin{rem}\label{opencondition}
Being $cu$-, $cs$-, or bi-inflatable, are $C^1$-open conditions on diffeomorphisms in $\DS$. 
\end{rem}

We say that an $f$-invariant ergodic measure is an ergodic component of $m$ if it is the restriction of $m$ to some set of positive $m$-measure. Any set $A$ giving rise to an ergodic component $m_A$ in this fashion is uniquely defined up to a set of zero $m$-measure. It is therefore quite harmless to refer to $A$ as an ergodic component. Sometimes we say that $A$ is a representation of $m_A$.
Pesin's spectral decomposition theorem \cite{MR0466791} implies that, for every $f \in \NA$, one can write $M$ as the union of a finite or countable number of pairwise disjoint $f$-invariant sets $A_i$, with $m(A_i)>0$, all of which represent ergodic components $m_{A_i}$. 

\begin{thm}[Main technical result] \label{main}
 Let $f$ be a non-uniformly Anosov diffeomorphism on $M$, and $A$ one of its ergodic components. 
Suppose that $A$ is $cu$-inflatable, then $m$-almost every $x\in A$ has an essential unstable  manifold of infinite span. If, moreover,  $\dim E^{cu} \leq 2$, then $m$-almost every $x\in A$ has a full unstable manifold of infinite span.
\end{thm}

\begin{rem} \label{no continuity}
 Notice that we have not claimed that the global unstable manifolds obtained in Theorem \ref{main} vary continuously in any sense, and it is not clear from our construction whether that is the case. The problem of continuity is intimately related to integrability of $E^{cu}$ and should explain why we need the assumption of integrability in Proposition \ref{thick} and Theorem \ref{staberg through transitivity}. 
\end{rem}

As an immediate consequence, we get:
\begin{cor} \label{at least one}
Suppose $f \in \NA$ is $cu$-inflatable. Then $f$ has at least one $cu$-inflatable ergodic component, $A$ say. Hence $m$-almost every $x\in A$ has an essential unstable manifold of infinite span.
\end{cor}

Given some bi-inflatable $f \in \NA$, we know from Corollary \ref{at least one} that it has at least one $cu$-inflatable ergodic component, say $A_{cu}$, and at least one $cs$-inflatable component, say $A_{cs}$. By the following proposition, the problem of ergodicity of $f$ reduces to establishing under what condition the two components coincide. 

\begin{prop}[First auxillary result] \label{bi-inflatable implies ergodicity}
Let $f\in \DS$ and suppose there exists a Borel set $A\subset M$ of positive Lebesgue measure with  $f_{\vert A}$ ergodic and such that, at $m$-almost every $x\in A$,  $\lambda_+^{cs}(x)<0$ and $\lambda_-^{cu}>0$. Suppose, furthermore, that for some $\delta>0$,  $m$-almost every $x\in A$ has an essential stable and unstable manifold of span $\delta$. Then $m(A) = 1$ so that $f$ is ergodic.
\end{prop}

We will soon give conditions that imply the existence of a bi-inflatable component, hence ergodicity. Once ergodicity is obtained, stable ergodicity follows automatically from the second auxiliary result. Denote by $\mathcal{E}$ the set of ergodic diffeomorphisms in $\Diff_m^2(M)$, by $\BI$ the subset of $\DS$, consisting of bi-inflatable diffeomorphisms.

\begin{prop}[Second auxillary result] \label{EBINA is open}
The set $\mathcal{E}\cap \BI \cap \NA$ is open in the $C^1$ topology. 
\end{prop}

\subsection{Stable ergodicity through transitivity}\label{transitivity}

It is an interesting problem to understand under what conditions transitivity of a conservative diffeomorphism implies that it is ergodic. In fact, it was conjectured by Tahzibi in \cite{T} that every transitive non-uniformly Anosov diffeomorphism is ergodic. Although the conjecture remains enigmatic in its full generality, our current approach sheds some light on the question in the case of inflatable systems. 

If $f\in \NA $ is bi-inflatable, it has an ergodic component $A_{cs}$ on which $m$-almost every point has a stable manifold of infinite span, and an ergodic component $A_{cu}$ on which $m$-almost every point has an unstable manifold of infinite span. If we knew that both $A_{cs}$ and $A_{cu}$ are open, then transitivity of $f$ would imply that $A_{cs}=A_{cu}$. Hence, by Proposition \ref{bi-inflatable implies ergodicity} and Proposition \ref{EBINA is open}, it would follow that $f$ is stably ergodic. Since $A_{cs}$ and $A_{cu}$ are saturated by stable/unstable manifolds of uniform span, there is great hope that they are indeed open, as was pointed out to me by Alexander Arbieto during an early exposition of this work. However, we do not know whether that is necessarily the case, the main difficulty being that we do not know whether, given $\delta>0$, the stable (unstable) manifolds of span $\delta$, associated to $x\in A_{cs}$ ($A_{cu}$) vary continuously with $x$ (see Remark \ref{no continuity}). The remedy is to assume a suitable kind of integrability of the invariant bundles.
The kind of integrability that in a natural way implies continuous dependence of essential stable and unstable manifolds is the one that the authors in \cite{MR2410949} call \emph{plaque unique integrability}. 

\begin{defn}\label{plaque unique}
 We say that a continuous distribution $E$ of $k$-planes in $M$ is plaque uniquely integrable if there exists a foliation $\mathcal{F}$ of $M$ into $k$-dimensional immersed manifolds tangent to $E$, and such that, if $D$ is any $k$-dimensional disk tangent to $E$, then $D$ is contained in some leaf of $\mathcal{F}$.
\end{defn}

It is shown in \cite{MR1432307} that this notion is stronger than simply asking for the uniqueness of integral foliation to a given continuous subbundle of $TM$. The good thing about having $E^{cu}$ plaque uniquely integrable is that the essential unstable manifold of a point $x\in M$ of span $\delta>0$ is the ball of radius $\delta$ around $x$ in the unique leaf of the foliation, containing $x$. 

\begin{prop}\label{thick} 
 Suppose that the central-unstable bundle of $f\in \NA$ is plaque uniquely integrable. Then every $cu$-inflatable ergodic component is open, modulo a set of zero Lebesgue measure.
\end{prop}

\begin{proof}
Let $A$ be a $cu$-inflatable ergodic component of some $f\in \NA$ for which $E^{cu}$ is plaque uniquely integrable. Fix $\delta>0$ and pick any $p \in M$ such that $m_{\Wlu(p)}$-almost every point $x \in \Wls(p)$ has an essential unstable manifold $\Gamma_x$ of span $\delta$ and such that $m_{\Wlu(p)}$-almost every $x \in \Wlu(p)$ is backward generic for $m_A$, in the sense that 
\begin{equation*}
 \varphi_-(x) := \lim_{n \rightarrow \infty} \frac{1}{n} \sum_{k=0}^{n-1} f^{-k}(x) = \int \varphi \ dm_A
\end{equation*}
for every continuous $\varphi: M \rightarrow \mathbb{R}$. (See Lemma \ref{disintegration} and Corollary \ref{full measure in component} for justification of why such $p$ exists.)

By plaque unique integrability of $E^{cu}$, to every $x\in \Wls(p)$ there is a unique immersed submanifold $F_x$ of $M$, passing through $x$, and tangent to $E^{cu}$. Write $F_x(\delta)$ for the ball, in $F_x$,  of radius $\delta$ around $x$ and $U_p = \bigcup_{x \in \Wls(p)} F_x(\delta)$. Then $U_p$ is an open set and we claim that $A$ has full  Lebesgue measure in $U_p$. Suppose that is not the case. Then there is some other ergodic component $B$ with positive Lebesgue measure in $U_p$. Pick $q \in M$ admitting a local unstable manifold such that $m_{\Wlu(q)}$-almost every $y \in \Wlu(q)$ has a local stable manifold $\Wls(y)$ on which $m_{\Wls(y)}$-almost every point is backward generic for $m_B$. Let $V_q$ be the union of local stable manifolds of such points. Now, by absolute continuity of Pesin's local stable manifolds, there is an open set $V\subset \Wls(p)$ such that $m_{F_x(\delta)} (V_q)>0$ for every $x \in V$. But for $m_{\Wls(p)}$-almost every $x$, $m_{\Gamma_x}(A) = \vert m_{\Gamma_x} \vert$ so that $m_{\Gamma_x}(V_p) \leq m_{\Gamma_x}(B) = 0$. Since $F_x(\delta)$ coincides with $\Gamma_x$ this contradicts the existence of $B$.
\end{proof}

\begin{cor}
 Let $f$ be a bi-inflatable non-uniformly Anosov diffeomorphism for which both the central-stable and the central-unstable bundles are plaque uniquely integrable. Then $f$ is stably ergodic.
\end{cor}

\begin{proof}
Any diffeomorphism $f$ satisfying the hypothesis of the corollary has at least one $cs$-inflatable ergodic component $A_{cs}$ and at least one $cu$-inflatable ergodic component $A_{cu}$. By Proposition \ref{thick}, both of them are open, modulo a set of zero Lebesgue measure. Transitivity of $f$ thus implies that $m(A_{cs}\cap A_{cu})>0$. But since $f_{\vert A_{cs}}$ and $f_{\vert A_{cu}}$ are ergodic, we must have $A_{cs} = A_{cu}$ up to a set of zero Lebesgue measure. Hence, by Proposition \ref{bi-inflatable implies ergodicity} and Proposition \ref{EBINA is open}, $f$ is stably ergodic.
\end{proof}

It turns out that it suffices to assume plaque unique integrability of one of the subbundles in order to obtain ergodicity and stable ergodicity. As usual we give only one of the statements, the other being completely analogous.

\begin{thm}\label{staberg through transitivity}
 Let $f$ be a transitive, $cu$-inflatable, non-uniformly Anosov diffeomorphism for which the  central-unstable 
bundle is plaque uniquely integrable. Then $f$ is ergodic. If in addition $f$ is $cs$-inflatable, then $f$ is stably ergodic.
\end{thm}

\begin{proof}
Let $f \in \NA $ be transitive, $cu$-inflatable, with plaque uniquely integrable central-unstable bundle. Then, by Theorem \ref{main} and Proposition \ref{thick}, there exists an ergodic component $A$, open up to a set of zero Lebesgue measure, on which Lebesgue almost every point admits an essential unstable manifold of infinite span. Suppose that $m(A)<1$, so that $f$ has some other ergodic component, $B$ say. Pick $p \in M$ admitting a local unstable manifold on which $m_{\Wlu(p)}$-almost every point has a local stable manifold and is forward generic for $m_B$. That is, 
\begin{equation*}
 \varphi_+ (x) := \lim_{n \rightarrow \infty} \frac{1}{n} \sum_{k=0}^{n-1} f^k (x) = \int \varphi \ dm_B
\end{equation*}
for every continuous $\varphi: M \rightarrow \mathbb{R}$ and $m_{\Wlu(p)}$-almost every $x \in \Wlu(p)$. Let $V_p$ denote the union of the local stable manifolds of such points. Since $A$ is open and dense, there exists a sequence of points $q_n \rightarrow p$ such that every $q_n$ has an essential unstable manifold $\Gamma_{q_n}$ of span $\delta$ for some fixed $\delta>0$ and such that $m_{\Gamma_{q_n}}(A) = \vert m_{\Gamma_{q_n}} \vert$. But it follows from absolute continuity of the stable foliation that, for all $n$ sufficiently large, we must have $m_{\Gamma_{q_n}}(B) \geq m_{\Gamma_{q_n}}(V_p)>0$, contradicting the existence of $B$.

If $f$ is $cs$-inflatable, it follows directly from Theorem \ref{main} and Proposition \ref{EBINA is open} that $f$ is stably ergodic.
\end{proof}

We do not know whether the plaque unique integrability assumption is necessary. It would be superfluous if the unstable essential manifolds vary continuously. In particular, it would be superfluous if it turns out that the invariant subbundles of a non-uniformly Anosov diffeomorphism are always plaque uniquely integrable. 

\begin{prob}
 Let $f$ be a non-uniformly Anosov diffeomorhpism. What can we say about the integrability of $E^{cs}$ and $E^{cu}$?
\end{prob}

\subsection{Stable ergodicity through Partial Hyperbolicity}

We say that $f \in \mathcal{DS}$ is partially hyperbolic if either $E^{cs}$ is uniformly contracting, i.e. there exists $C>0$, $0<\tau<1$ such that 
\begin{equation*}
 \| Df_{\vert E_x^{cs}}^n \| \leq C \tau^n \quad \forall x \in M, \ n \geq 0,
\end{equation*}
or if $E^{cu}$ is uniformly expanding, i.e. there exists $C>0$, $0<\tau<1$ such that
\begin{equation*}
 \| Df_{\vert E_x^{cu}}^{-n} \| \leq C \tau^n \quad \forall x\in M, \ n \geq 0.
\end{equation*}

The two notions are interchanged upon replacing a diffeomorphism by its inverse, and we shall therefore content ourselves by consider the former. In this case we write $TM = E^s \oplus E^{cu}$ for the partially hyperbolic splitting.

It is well-known (see \cite{MR0501173}) that if $f \in \mathcal{DS}$ has a partially hyperbolic splitting $TM = E^s \oplus E^{cu}$, then $E^s$ is plaque uniquely integrable and every $x \in M$ has a local stable manifold of uniform span. The existence of a $cu$-inflatable ergodic component is therefore enough to obtain ergodicity. From Corollary \ref{at least one}, together with the first and second auxiliary results (Proposition \ref{bi-inflatable implies ergodicity} and Proposition \ref{EBINA is open}) we get:

\begin{thm} \label{PH inflatable implies staberg}
 Suppose that $f \in \NA$ is $cu$-inflatable and admits a partially hyperbolic splitting $E^s \oplus E^{cu}$. Then $f$ is stably ergodic.
\end{thm}

\subsection{Stable ergodicity through global product structure}

Another way to guarantee the existence of a bi-inflatable ergodic component for a diffeomorphism $f \in \NA$ is by requiring global product structure. It is a useful condition when constructing examples, because it is checkable for derived-from-Anosov diffeomorphisms. 

\begin{defn}
 We say that $f \in \mathcal{DS}$ has global product structure if there exists $K>0$ such that, given any $cs$-dimensional disk $D_1$ of tangent to $E^{cs}$ of span $K$, and any $cu$-dimensional disk $D_2$ tangent to $E^{cu}$ of span $K$, we have $D_1 \cap D_2 \neq \emptyset$.
\end{defn}

\begin{thm}
 Let $f$ be a bi-inflatable non-uniformly Anosov diffeomorphism with global product structure. Suppose, moreover, that either $\dim E^{cs}$ or  $\dim E^{cu}$ (or both) is at most bi-dimensional. Then $f$ is stably ergodic.
\end{thm}

\begin{proof}
 Let $f$ be as in the theorem, and let $A_{cs}$ and $A_{cu}$ be $cs$- and $cu$-inflatable components of $f$. Our aim is to show that $A_{cs} = A_{cu}$ up to a set of zero Lebesgue measure. Suppose, without loss of generality, that $\dim E^{cs} \leq 2$.  Then there exists (see Lemma \ref{disintegration} and Corollary \ref{full measure in component} for justification) some $p \in M$, admitting a local unstable manifold $\Wlu(x)$ in which there is a full $m_{\Wlu(p)}$-measure set $\Lambda$ such that  every $x \in \Lambda$ is forward generic for $A_{cs}$ and has a full stable manifold $\Gamma_x^s$ of span $K$. Moreover, there exists $q$ admitting an essential unstable manifold $\Gamma_q^u$ of span $K$ and such that $m_{\Gamma_q^u}$-almost every $y \in \Gamma_q^u$ is forward generic for $A_{cu}$. By the local product structure, every $\Gamma_x^s, \ x \in \Lambda$ intersects $\Gamma_q^u$ in at least one point. In particular, there exists some $n\geq 0$ such that 
\begin{equation*}
 \Lambda_n := \{x \in \Lambda: f^{-n}(\Wls(f^n(x))) \cap \Gamma_q^u \neq \emptyset \}
\end{equation*}
has positive $m_{\Wlu(p)}$-measure. By absolute continuity of Pesin's local stable manifolds, this implies that 
\begin{equation*}
m_{\Gamma_q^u} \left( \bigcup_{x \in \Lambda_n}  f^{-n}(\Wls(f^n(x))) \right)>0,
\end{equation*}
 showing that $A_{cs}= A_{cu}$. Hence by the first and second auxiliary results (Proposition \ref{bi-inflatable implies ergodicity} and \ref{EBINA is open}), $f$ is stably ergodic.
\end{proof}

\begin{example}
In order to bring more life to the notion of global product structure, we describe a concrete situation in which it is present. Suppose that $f: \mathbb{T}^n \rightarrow \mathbb{T}^n$ has a dominated splitting 
$T \mathbb{T}^n = E^{cs}\oplus E^{cu}$ with the following property: there exist constant conefields $S^{cs}, S^{cu} \subset T \mathbb{T}^n$, containing $E^{cs}$ and $E^{cu}$ respectively, such that $S_p^{cs} \cap S_p^{cu} = \{0\}$ for some (hence all) $p \in \mathbb{T}^n$. Then clearly $f$ has a global product structure.
\end{example}

\section{Proof of the main technical result}

We dedicate the whole of this section to the proof of Theorem \ref{main}. We start by describing a way of decomposing Lebesgue measure into smooth measures along local stable and unstable manifolds, quite different from, and more suitable to our needs, than the well known disintegration of Rohlin into measures on elements of measurable partitions. The reason for doing this is that the families $\mathcal{W}^{\sigma} = \{W_{loc}(x)^{\sigma} : x \in H \}, \ \sigma =s, u, $ are not partitions (if $y \in \Wlu(x)$, then 
$\Wlu(x) \cap \Wlu(y) \neq \emptyset$ but $\Wlu(x) \neq \Wlu(y)$). A natural disintegration of $m$ into smooth measures on local unstable manifolds should therefore not yield a family of mutually singular measures.

\subsection{Non-singular disintegration}

Recall that $H$ denotes the set of hyperbolic points, i.e. those $x \in M$ for which (\ref{contraction}) and (\ref{expansion}) hold.
For $x \in H$ we write $m_x^u = m_{\Wlu(x)}$.  Let $\M(M)$ denote the space of positive finite Borel measures on $M$, endowed with the weak*-topology. Consider some $\nu \in \M(M)$ and a mapping $M \ni x \mapsto \nu_x \in \M(M)$ such that for every continuous $\varphi:M \rightarrow \mathbb{R}$, $ x \mapsto \int \varphi d\nu_x$ is measurable. We can then define the integral $\int \nu_x d\nu(x)$ to be the unique measure $\mu$, satisfying
\begin{equation}\label{integration of measures}
 \int \varphi d\mu = \int \left( \int \varphi d\nu_x \right) d\nu(x), \quad \forall \ \varphi \in C^0(M, \mathbb{R}).
\end{equation}

\begin{prop}\label{disintegration}
 There exists a family of measures $\{ \tilde{m}_x^u : x \in M \}$, with 
\begin{itemize}
 \item $\tilde{m}_x^u \sim m_x^u$ for $m$-almost every $x$,
\item $ \int \tilde{m}_x^u \ dm(x) = m$.
\end{itemize}
\end{prop}

Clearly the same result is true upon replacing $u$ with $s$.

\begin{proof}
 Let $\mu := \int m_x^u dm(x)$. As a consequence of absolute continuity of Pesin's unstable holonomies, we have $\mu \sim m$. Hence, by the theorem of Radon Nikodym, there exists $\rho \in L^1(M, m)$, positive $m$-almost everywhere, such that
\begin{equation*}
 \mu(E) = \int_E \rho \ dm \ \text{ for every Borel set } E \subset M.
\end{equation*}

For $x\in H$, let $\phi_x : \Wlu(x) \rightarrow \mathbb{R}$ to be the restriction of $\rho^{-1}$ to $\Wlu(x)$ and take $\tilde{m}_x^u = \phi_x m_x^u$. Note that the family $\phi_x$ is well defined in the following sense: If $\rho' = \rho$ $m$-almost everywhere, and $ {\phi_x}' = {\rho'}^{-1}_{\vert \Wlu(x)}$, then ${\phi_x}' = \phi_x$ $m_x^u$-almost everywhere, for $m$-almost every $x$. Indeed, the formula
\begin{equation*}
\int m_x^u(\{\rho' \neq \rho \}) dm(x) = \mu(\{\rho' \neq \rho \}) = 0
\end{equation*}
implies that $m_x^u (\{\rho' \neq \rho \}) = 0$ $m$-almost everywhere.

Our choice of the family $\phi_x$ now immediately gives that
\begin{align}
m(E) & = \int \chi_E dm = \int \frac{\chi_E}{\rho} d\mu  = \int
\left( \int \frac{\chi_E(y)}{\rho(y)} dm_x^u(y) \right) dm(x) \\ &
= \int \left(\int_E \phi_x dm_x^u \right) dm(x) = \int \tilde{m}_x^u(E) \ dm(x) \end{align} 
for every Borel set $E \subset M$.
\end{proof}

\begin{rem}
 Proposition \ref{disintegration} may be seen as an adaptation of the disintegration technique used in \cite{Andersson}. However, the proof given here is more concise and calls for abstraction: if $ M \ni x \mapsto \nu_x \in \M(M)$ is a measurable family such that $\int \nu_x \ d\nu(x) \sim \nu$, then there exists a measurable family $M\ni x \mapsto \tilde{\nu}_x \in \M(M)$, such that $\tilde{\nu}_x \sim \nu_x$ $\nu$-almost everywhere, and $\int \tilde{\nu}_x \ d\nu(x) = \nu$ (same proof). 
\end{rem}

\begin{cor}\label{full measure almost everywhere}
Let $f \in \NA$, and let $S$ be any Borel set of full Lebesgue measure. Denote by $S^*$ the set of points $x \in M$ such that $\Wlu(x)$ intersect $S$ in a set of full $m_x^u$-measure. Then $m(S^*)=1$.
\end{cor}

\begin{proof}
 Proposition \ref{disintegration} gives
\[ m(S) = \int \tilde{m}_x^u(S) dm(x) = 1,
\]
whence it follows that $\tilde{m}_x^u(S) = \vert \tilde{m}_x^u \vert$ for $m$-almost every $x \in H$. The statement follows since $\tilde{m}_x^u \sim m_x^u$ for $m$-almost every $x \in M$.
\end{proof}

\begin{cor}\label{full measure in component}
 Let $f \in \NA$ and suppose that $A\subset M$ is (a representative of) an ergodic component of $f$. Let $B$ be any Borel set with 
$m(A \bigtriangleup B)=0$. Then $m_x^u(B) = \vert m_x^u \vert$ for $m$-almost every $x\in A$.
\end{cor}

\begin{proof}
 Consider the Borel set 
\begin{equation*}
 A_- = \{x\in M : \lim_{n \rightarrow \infty} \frac{1}{n} \sum_{k=0}^{n-1} \varphi (f^{-k}(x)) = \int \varphi \ dm_A
\  \forall  \varphi \in C^0( M, \mathbb{R}) \}.
\end{equation*}
By ergodicity of $m_A$ we have $m(A \bigtriangleup A_-) = 0$, so that $A_-$ is just another representative of $A_-$ with the special property that $\Wlu(x) \subset A_-$ for every $x \in A_-$. In particular, $m_x^u(A_-) = \vert m_x^u \vert$ for every $x \in A_-$. Applying Proposition \ref{disintegration}, we get 
\begin{equation*}
 0 = m(B \bigtriangleup A_-) = \int \tilde{m}_x^u (B \bigtriangleup A_-) dm(x),
\end{equation*}
so that $\tilde{m}_x^u (B \bigtriangleup A_-)= 0$ at $m$-almost every $x \in M$. In particular, $\tilde{m}_x^u(B) = \vert \tilde{m}_x^u \vert$ for $m$-almost every $x\in A$. The proof follows since $m_x^u \sim \tilde{m}_x^u$ for $m$-almost every $x\in M$.

\end{proof}

\subsection{A geometric estimate}

At the heart of the proof of Theorem \ref{main} lies a simple, yet powerful, geometric estimate. Its intuitive content is that if the volume of a manifold with boundary is much larger than the volume of its boundary, then most points in the manifold see nothing or a very small portion of the boundary in its proximity. Despite its simplicity, it constitutes the most important step in understanding how inflatablility yields essential stable or unstable manifolds of a desired span.

Recall from section \ref{statements} that we take Riemannian manifold to mean a $C^1$ manifold with a continuous metric.

\begin{defn}
 Let $D$ be a $d$-dimensional Riemannian manifold, possibly with boundary. 
We say that $K>0$ is a bound for the geometry of $D$, at level $\delta>0$, if 
$m_D (B_{\delta}^D(p)) \leq K$ for every $p \in D$.
\end{defn}

\begin{prop} \label{geometric}

Let $D$ be a compact manifold with boundary, whose geometry at level $\delta$ is bounded by $K$. Denote by $\theta: D \rightarrow \mathbb{R}$ the map $p \mapsto m_{\partial D} (B_{\delta}^D(p))$ i.e. that which measures how much of the boundary of $D$ is seen in the ball of radius $\delta$, centered at $p$. Then we have
\begin{equation*} \int \theta  \ dm_D  \leq K 
\vert m_{\partial D} \vert. 
\end{equation*} 
\end{prop}

\begin{proof}
The proof consists of one simple trick. By applying Fubini's theorem on the product space $(D \times \partial D, m_D \times m_{\partial D})$
we obtain 
\begin{align}
 \int \theta  \ dm_D  &= \int \left( \int \chi_{ \{(x,y) \in D
\times
\partial D : d(x, y) < \delta \} }  dm_{\partial D}(y) \right) dm_D(x) \\ &= \int
\left( \int \chi_{ \{(x,y) \in D \times \partial D : d(x, y) <
\delta \} }  dm_D(y) \right) dm_{\partial D}(x) \\
&= \int m_D (B_{\delta}(y)) dm_{\partial D}(y) \leq K 
\vert m_{\partial D} \vert\end{align}
as required.
\end{proof}

\subsection{Small sight of the boundary}\label{small sight}

It is a consequence of the compactness of $M$ and continuity of the $E^{\sigma}$, $\sigma =s,u$, that, given any $\delta>0$, there exists $K>0$ such that every manifold $D$ (with or without boundary), tangent to either of $E^{cs}$ or $E^{cu}$, has geometry bounded by $K$ at level $\delta$.
Given any $x \in H$ and $n \geq 0$, we write $ D_{x,n} = \overline{f^n(\Wlu(x))}$, $ D_{x,n}^{\circ} = f^n(\Wlu(x))$, and define
\begin{eqnarray}
\theta_{x,n}: D_{x,n} & \rightarrow & \mathbb{R} \\
      p & \mapsto & m_{\partial D_{x,n}} (B_{\delta} (p)).
\end{eqnarray}
For any real $h>0$, let $G_{x,n,h} = \{y \in D_{x, n}^{\circ} : \theta_{x,n}(y) < h \}$. When considering a small value of $h$, points in $G_{x,n,h}$ are good  in the sense that their $\delta$-neighbourhood see only a small piece of  $\partial D_{x,n}$. Consequently, the open ball $B_{\delta}^{D_{x,n}^{\circ}}(p)$, centered at $p \in G_{x,n,h}$, is nearly a `full disk', consisting only of points in the global unstable manifold $\Wsu(p)$ of $p$. 

To prove Theorem \ref{main} we define a set of full Lebesgue measure in $A$ on which every point admits an essential unstable manifold of uniform span. The set we consider is
\[G = \bigcap_{h>0}  \bigcup_{n \geq 0}  \bigcup_{x \in H \cap A} G_{x,n,h}.
\]

\begin{lem}\label{full measure}
 The set $G$ has full Lebesgue measure in $A$.
\end{lem}

Before proving Lemma \ref{full measure} we establish an auxiliary result. Write 
\begin{equation*}
\Xi_n^{cu} (f) := \log \left( \sup_{x \in M} \| \wedge^{cu-1}Df_{\vert E_x^{cu}}^n \| \right).
\end{equation*}
Notice that $\Xi_n^{cu}(f)$ is sub-additive, i.e. $\Xi_{n+m}^{cu}(f) \leq \Xi_n^{cu}(f) + \Xi_m^{cu}(f)$ for all $n,m \geq 0$.

\begin{lem}\label{maximum growth}
Given any $x \in H$ and $h>0$ we have
\[\limsup_{n \rightarrow \infty} \frac{1}{n} \log m_{D_{x,n}}(G_{x,n,h}^c)
\leq \inf_{n\geq 0} \frac{\Xi_n^{cu}(f)}{n} .\]
\end{lem}

\begin{proof}
 By Chebychev's inequality we have 
\[m_{D_{x,n}}(G_{x,n,h}^c) 
\leq \frac{1}{h} \int \theta_{x,n} \ dm_{D_{x,n}} \leq \frac{K}{h} \vert m_{\partial D_{x,n}} \vert.\] The proof follows since $\vert m_{\partial D_{x,n}} \vert \leq \vert m_{\partial D_{x,0}} \vert e^{\Xi_n^{cu}(f)}$.
\end{proof}

\begin{proof}[Proof of Lemma \ref{full measure}]

We write $F_{h,n,x} = f^{-n} (G_{h,n,x})$ and 
\begin{equation*} 
F_{h, n}  = \bigcup_{x \in H \cap A} F_{h,n,x}, \quad 
F_h  = \bigcup_{n \geq 0} F_{h,n}, \quad
F  = \bigcap_{h >0} F_h.
\end{equation*}

Notice that since $f$ is conservative, we have $m(F_{h,n}) = m( f^{-n} G_{h,n}) = m(G_{h,n})$. So in order to prove the Lemma, it is sufficient to establish that, for all $h>0$,  
\begin{equation}\label{limit}
\lim_{n \rightarrow \infty} m(F_{h,n}) = m(A).
\end{equation}
For then $m(F_h) = m(A)$ for every $h>0$ and the proof follows by intersecting $F_h$ over all positive rational values of $h$. Furthermore, (\ref{limit}) will be proved once we have shown that, given any $h>0$, and $m$-almost every $x \in H\cap A$, we have 
\begin{equation}\label{limit is one}
\lim_{n \rightarrow \infty} m_x^u( F_{h,n,x}) = \vert m_x^u \vert.
\end{equation}
For then we may use the disintegration from Proposition \ref{disintegration} to conclude that
\begin{align}
 m(F_{h,n}) & = \int \tilde{m}_x^u(F_{h,n}) \ dm(x) \geq \int_A \tilde{m}_x^u (F_{h,n,x}) \ dm(x) \\
 & \rightarrow \int_A \vert \tilde{m}_x^u \vert dm(x) = m(A) \text{ as } n \rightarrow \infty,
\end{align}
by the Dominated Convergence Theorem. 

We will prove (\ref{limit is one}) by contradiction. But first let us determine the set of points on which (\ref{limit is one}) holds. For $n \geq 0$ write $\xi_n^{cu}(f, x) = \log \vert \det Df_{\vert E_x^{cu} }^n \vert$ and let 
\begin{equation*}
B = \{x \in A \cap H: \lim_{n \rightarrow \infty} \frac{1}{n}  \xi_n^{cu}(x) = \int \xi_0^{cu} dm_A \}.
\end{equation*} 
By Corollary (\ref{full measure in component}), the set $B^* = \{x \in B: \tilde{m}_x^{cu}(B) = \vert \tilde{m}_x^{cu} \vert \}$ has full Lebesgue measure in $A$. The claim is that (\ref{limit is one}) holds for every $x \in B^*$. Suppose it is false. Then, noticing that $\{F_{h,n,x}^c \}_{n\geq0}$ is a decreasing sequence in $n$, we must have hat $m_x^u(E_{h,x})>0$ for some $x\in B^*$, where $E_{h,x} = \bigcap_{n\geq 0} F_{h,n,x}^c$. We simplify notation in the next calculation by writing $\mu = m_x^u$, $E = E_{h,x}$, and $\mu_E$ for the normalised restriction of $\mu$ to $E$. Application of Jensen's inequality gives

\begin{align} \liminf_{n \rightarrow \infty} \frac{1}{n} \log m_{D_{x,n}}(G_{x,n,h}^c) 
& \geq \liminf_{n \rightarrow \infty} \frac{1}{n} \log \int_E \vert \det Df_{\vert E_y^{cu} }^n \vert \ d\mu(y) \\
& \geq \liminf_{n \rightarrow \infty}  \int \frac{1}{n} \xi_n^{cu} \ d\mu_E + \frac{1}{n} \log \mu(E) \\
 & \geq \int \xi_0^{cu} \ dm_A> \inf_{n\geq 0} \frac{\Xi_n^{cu}(f)}{n}, 
\end{align}
contradicting Lemma \ref{maximum growth}.

\end{proof}

\subsection{Construction of essential unstable manifolds}

In this section we prove that every $x \in G$ has an essential unstable manifold $\Gamma_x$ of uniform span $\delta$, for some small $\delta>0$. Once that is done, it is a small step to show that $m$-almost every $x \in A$ has an essential unstable manifold of infinite span. Indeed, we may write $G^1 = \{x \in G: m_{\Gamma_x}(G) = \vert m_{\Gamma_x} \vert \}$, $G^2 = \{x \in G^1: m_{\Gamma_x}(G^1) = \vert m_{\Gamma_x}\vert \}$ etc., and $G^{\infty} = \bigcap_{n \geq 1} G^n$. By Corollary \ref{full measure in component}, each $G^n$ has full Lebesgue measure in $A$, and so has $G^{\infty}$. Now pick $x \in G^{\infty}$. Then 
\begin{equation*}
\Gamma_x^{(1)} : = \bigcup_{y \in \Gamma_x \cap G^{\infty}} \Gamma_y 
\end{equation*}
is an essential unstable manifold of span $2 \delta$, 
\begin{equation*}
\Gamma_x^{(2)} = \bigcup_{y \in \Gamma_x^{(1)} \cap G^{\infty} } \Gamma_y 
\end{equation*}
 is an essential unstable manifold of span $3 \delta$ etc.

We show that every $x\in G$ has an essential unstable manifold of uniform span by taking some small $\rho>0$ and defining a $C^1$ map $\psi_x$ from $E_x^{cu} (\rho)$ (the ball of radius $\rho$ in $E_x^{cu}$, centered at the origin) to $E_x^{cs}$ such that $\Gamma_x := \exp_x( \graph \psi_x )$ is an essential unstable manifold. All we require from $\rho$ is that $\graph \psi_x$ must be in the domain of definition of $\exp_x$ and will depend only on how much $E^{cu}$ varies in exponential charts.

Denote by $\pi: E_x^{cu} \times E_x^{cs} \rightarrow E_x^{cu}$ the projection onto the first coordinate, and by $\mathcal{C}_x$ the deformed cylinder $\exp_x (E_x^{cu} (\rho) \times E_x^{cs})$. 

Fix $h$ a bit larger than $\rho$. From the definition of $G$, there is some sequence of non-negative integers $n_k$, $k\geq 0$ and points $x_k \in M$ such that 
\begin{equation*}
m_{\partial D_{x_k, n_k}} (B_{h}^{D_{x_k, n_k}}(x)) \rightarrow 0 \text{ as } k \rightarrow \infty.
\end{equation*}
Denote by $\gamma_k$ the connected component of $f^{n_k}(\Wlu(x_k)) \cap \mathcal{C}_x$ containing $x$, and by $U_k$ the set $\pi(\exp_x^{-1} \gamma_k )$. Since $d(f^{-n}(x_k), f^{-n - n_k}(x)) \rightarrow 0$ exponentially fast, and the size of local unstable manifolds vary subexponentially along orbits, there exists, for each $x_k$, an integer $N_k$, such that $f^{-N_k}(\Wlu(x_k)) \subset \Wlu(f^{-N_k - n_k}(x))$. Let $\tilde{\gamma}_k$ be the connected component of $f^{N_k + n_k}(\Wlu(f^{-N_k - n_k}(x))\cap \mathcal{C}_x$ containing $x$. Then $\tilde{\gamma}_k \supset \gamma_k$ so that $\tilde{U}_k:= \pi(\exp_x^{-1} \tilde{\gamma}_k ) \supset U_k$.

Write $U = \bigcup_{k \geq 0} U_k \subset E_x^{cu}(\rho)$ and $\tilde{U} = \bigcup_{k \geq 0} \tilde{U}_k$. For each $k$ and each $z \in \tilde{U}$ we define $\psi_x^k:\tilde{U}_k \rightarrow E_x^{cs}$ by the relation 
\[ \exp_x ( (z, \psi_x^k (z))) \in \tilde{\gamma}_k.\]
By definition, $\tilde{U}_{k+1} \supset \tilde{U}_k$ and ${\psi_x^{k+1}}_{\vert \tilde{U}_k} = \psi_x^k$ for all $k \geq 0$.
Hence we may define $\psi_x$ on $\tilde{U} = \bigcup_{k\geq 0} \tilde{U}_k \supset U$ by $\psi_x (z) = \psi_x^k$ for any $k$ such that $z \in \tilde{U}_k$.

\begin{description}
\item[Claim 1] $U$ (and hence $\tilde{U}$) has full Lebesgue measure in $E_x^{cu}(\rho)$.
\item[Claim 2] $\psi_x$ admits a $C^1$ extension on $E_x^{cu}(\rho)$.
\end{description}

To prove the first claim, choose some hyperplane $P \subset E_x^{cu}$ passing through the origin. Denote by $\pi_P: E_x^{cu}  \rightarrow P$ the orthogonal projection onto $P$. Now, each $\partial U_k$ is a $cu-1$ dimensional submanifold of $E_x^{cu}(\rho)$ whose $cu-1$ dimensional volume goes to zero as $k \rightarrow \infty$. Since the Jacobian of ${\pi_P}_{\vert \partial U_k} : \partial U_k \rightarrow P$ is at most $1$, the $cu-1$ dimensional volume of $\pi_P (\partial U_k)$ in $P$ goes to zero as $k \rightarrow \infty$. Clearly $E_x^{cu}(\rho) \setminus U_k \subset \pi_P^{-1} (\pi_P (\partial U_k)) \cap (E_x^{cu}(\rho))^c$, and the latter has $cu$ dimensional volume bounded by $2 \rho$ times the $cu-1$ dimensional volume of $\pi_P (\partial U_k)$.

To prove the second claim, pick any $p \in E_x^{cu}(\rho)$, and any $\epsilon>0$. We will show that by choosing a sufficiently small neighbourhood $V$ of $p$ $E_x^{cu}$, we have

\begin{align}
 & \sup_{k \geq 0 } \ \sup_{u, v \in V \cap U_k} \| \psi_x^k(u)-\psi_x^k(v) \|  < \epsilon \label{cont 1} \\
 & \sup_{k \geq 0 } \ \sup_{u, v \in V \cap U_k} \| D \psi_x^k(u)- D \psi_x^k(v) \|  < \epsilon. \label{cont 2}
\end{align}

That implies that $\psi_x$ extends to a continuous function $\overline{\psi}_x: E_x^{cu}(\rho) \rightarrow E_x^{cs}$, and that $D\psi_x$ extend continuously to a continuous function $\Psi_x : E_x^{cu}(\rho)  \rightarrow L(E_x^{cu},E_x^{cs})$. It does not follow directly that $\Psi_x$ is the derivative of $\psi_x$, but it can be established by showing that, given any $p \in E_x^{cu}(\rho)$, and any $\epsilon>0$, there is a neighbourhood $V$ of $p$ such that 
\begin{equation}\label{cont 3}
 \sup_{k \geq 0} \ \sup_{u,v,w \in V\cap U_k} \frac{ \| \psi_x^k(u) - \psi_x^k(v) - \Psi_x(w)(u-v)\|}{\|u-v\|} < \epsilon. 
\end{equation}

Since $\exp_x( \graph \psi_x )$ is tangent to the continuous bundle $E^{cu}$, (\ref{cont 1}), (\ref{cont 2}), and (\ref{cont 3}) follows if we can prove that, given any $u,v \in U$, and any $\alpha>0$, there is a piecewise differentiable curve from $u$ to $v$, contained entirely in $U$, whose length is at most $\|u-v\| + \alpha$.  To this end, pick some $k_0$ sufficiently large so that $u, v \in U_{k_0}$. Let $B_u, B_v$ be balls of equal radius $\leq \alpha/2$ around $u$ and $v$, contained in $U_{k_0} \cap V $. Let $H$ be the orthogonal complement to $\{ \lambda (u-v): \lambda \in \mathbb{R} \}$ in $E_x^{cu}$, and $\pi : E_x^{cu} \rightarrow H$ the projection onto $H$.  Since the $cu-1$ dimensional volume of $\pi(\partial U_k)$ is at most equal to the $cu-1$ dimensional volume of $\partial U_k$, there is some $k \geq k_0$ such that the $cu-1$ dimensional volume of $\pi(\partial U_k)$ is strictly smaller than the $cu-1$ dimensional volume of $\pi(B_u)$ (which is the same as that of $\pi(B_v)$). That means that there exists a line segment from some point in $B_u$ to some point in $B_v$, parallel to $\{ \lambda (u-v): \lambda \in \mathbb{R} \}$, and entirely contained in $U_k$. By joining its endpoints with $u$ and $v$ by straight line segments, we have constructed a piecewise differentiable curve from $u$ to $v$ of length less than $\|u-v\| + \alpha$, thus concluding the proof of the second claim.
 
\begin{rem}
 The essential unstable manifolds obtained in the construction are $C^2$ on an open and dense subset of full leaf volume, since here they coincide with some iterate of a local unstable manifold on this set. However, it is not clear form the construction whether the full essential unstable manifold obtained is $C^2$ embedded, and it would be interesting to know whether or not that is the case.
\end{rem}

\subsection{Full unstable manifolds}

The last assertion in Theorem \ref{main} is that if $\dim E^{cu} \leq 2$, then $m$-almost every $x \in A$ has a full unstable manifold of infinite span. 
That is, not only do we have $m_{\Gamma_x}(\Wsu(x)) = \vert m_{\Gamma_x} \vert$, but actually $\Gamma_x \subset \Wsu(x)$. In particular, the unstable manifolds of span $\delta$ are $C^2$ in this case. For $\dim E^{cu} = 1$ this is trivial. For $\dim E^{cu} = 2$ it follows from the following observation. Let $D$ be a $2$-dimensional manifold with boundary such that $D$ has bounded geometry. Then, given any $\epsilon >0$, there exists $\beta>0$ such that if  $m_{\partial D}( B_{\delta}^D(p)) < \beta$, then $d^D(p, \partial D) > \delta - \epsilon$. 

The same thing is not true when $\dim E^{cu} \geq 3$. To see why, let $D_0$ be the closed unit ball in $\mathbb{R}^3$, centered at the origin. For every $\epsilon>0$ we can obtain a manifold with boundary $D_{\epsilon}$ by making small wormholes in $D_0$, approaching the origin such that $m_{\partial D_{\epsilon}}(B_1^{\mathbb{R}^3}(0))<\epsilon$ and $d(0, \partial D_{\epsilon})<\epsilon$. Still, it seems like a difficult task to understand whether this geometrical property can actually present itself upon iterating a local unstable manifold of a non-uniformly Anosov diffeomorphism.

\section{Proofs of auxiliary results}

The proof of the first auxiliary result (Proposition \ref{bi-inflatable implies ergodicity}) relies on the following modification of the Hopf-Anosov argument.

\begin{lem}\label{one large component implies local ergodicity}
Let $f \in \DS$ and suppose there is a Borel set $A \subset M$ with positive Lebesgue measure such that $f_{\vert A}$ is ergoidc, and such that $\lambda_+^{cs}(x)<0$ and $\lambda_-^{cu}(x) >0$ at $m$-almost every $x \in A$. If there is some $\delta >0$ such that $m$-almost every $x\in A$ has  essential stable and unstable
manifolds of span $\delta>0$, then there exists $\rho>0$ such that, for $m$-almost every $x \in A$, we have $m(B_{\rho}(x) \cap A^c) = 0$.
\end{lem}

For the proof of Lemma \ref{one large component implies local ergodicity}, it is convenient to represent the ergodic component $m_A$ by a set with some specially nice properties. Let $A_0$ be the set of points which are forward and backward generic for $m_A$, intersected with $G^{\infty}$, so that each $x \in A^0$ admits essential stable and unstable manifolds $\Gamma_x^s, \ \Gamma_x^u$ of span $\delta$, say, where $\delta>0$ is some fixed constant. We then define $A^1 = \{ x \in A^0: m_{\Gamma_x^s}(A^0) = \vert m_{\Gamma_x^s}\vert \}$, $A^2 = \{ x \in A^1: m_{\Gamma_x^u}(A^1) = \vert m_{\Gamma_x^u}\vert \}$, $A^3 = \{ x \in A^2: m_{\Gamma_x^s}(A^2) = \vert m_{\Gamma_x^s}\vert \}$ etc. Then, by Corollary \ref{full measure in component}, $A^{\infty} = \cap_{n \geq 0} A^n$ is a Borel set with $m(A^{\infty} \bigtriangleup A) =0$. It has the convenient property that if we pick any $x_0 \in A^{\infty}$, then $m_{\Gamma_{x_0}^u}$-almost every point in $\Gamma_{x_0}^u$ lies in $A^{\infty}$. Hence is generic for $m_A$ and has an essential stable and unstable manifold of span $\delta$. And if we pick any $x_1 \in \Gamma_{x_0}^u \cap A^{\infty}$, then $m_{\Gamma_{x_1}^s}$-almost every point in $\Gamma_{x_1}^s$ lies in $A^{\infty}$. Hence is generic for $m_A$ and has an essential stable and unstable manifold of span $\delta$. So it goes on eternally.

\begin{proof}[Proof of Lemma \ref{one large component implies local ergodicity}]
We shall show that there is $\rho>0$ such that $A^{\infty}$ has full Lebesgue measure in $B_{\rho}(p)$, for every $p \in A^{\infty}$. Fix  $p\in A^{\infty}$ arbitrarily. We choose $\rho>0$ sufficiently small so that there is a chart $\varphi : (-2,2)^d \subset \mathbb{R}^n \rightarrow M$, with $\varphi(0) = p$, such that the image of $(-1,1)^{d}$ contains $B_{\rho}(p)$, and such that the restriction of the subbundles $E^{cs}, \ E^{cu}$ to the image of $\varphi$ correspond to planes nearly parallel to the `horizontal' $\mathbb{R}^{cs} \times \{0\}$ and the `vertical' $\{0\} \times \mathbb{R}^{cu}$, meaning that they are graphs of linear maps
\begin{equation*}
 A^{cs}: \mathbb{R}^{cs} \rightarrow \mathbb{R}^{cu}, \quad A^{cu}: \mathbb{R}^{cu} \rightarrow \mathbb{R}^{cs}, 
\end{equation*}
where $A^{cs}$ and $A^{cu}$ have (Euclidian) norm less than, say, $\frac{1}{10}$. We also suppose that $\varphi((-2,2)^d)$ is small enough so that if $x \in (-1,1)^d$, then $\varphi^{-1} ( \Gamma_{\varphi(x)}^u)$ is the graph of a $C^1$ map $\psi_x: (-2,2)^{cu} \rightarrow (-2,2)^{cs}$. The same thing is required for essential stable manifolds. By compactness of $M$ and continuity of $E^{cs},\ E^{cu}$, such $\rho$ can be taken uniform, i.e. independent of $p \in M$. Altough we work in a chart, we shall use the same notation for subsets of $M$ and their counterparts in $\mathbb{R}^d$. Thus we write $A^{\infty}, \Gamma_x^s \subset \mathbb{R}^d$ instead of $\varphi^{-1} (A^{\infty}), \varphi^{-1}(\Gamma_x^s)$ etc.

We argue by contradiction. Suppose that $A^{\infty}$ does not have full measure in $(-1,1)^{cs}\times (-1,1)^{cu}$. Then, by Corollary \ref{full measure in component}, there is some $q$ in $(-1,1)^{cs} \times (-1,1)^{cu}$ such that $m_{\Wlu(q)}(B \cap H) = \vert m_{\Wlu(q)} \vert$, where $B$ represents an ergodic component different from $m_A$. Let 
\begin{equation*}
V_q = \bigcup_{x\in \Wlu(q)} \Wls(x) 
\end{equation*}
To arrive at a contradiction, it is enough to show that there exists $y \in (-1,1)^d\cap A^{\infty}$ close enough to $q$ that $m_{\Gamma_y^u}(C_p)>0$. Hence the proof is complete once we show that $A^{\infty}$ is dense in $(-1,1)^d$.

We define a sequence of sets $\{X_n : n\geq 0 \}$ recursively. Let $X_0 = \Gamma_p^u \cap A^{\infty}$. If $n\geq 1$ is odd, let $X_n = A^{\infty} \cap \bigcup_{x\in X_{n-1}} \Gamma_x^s$. If $n\geq 2$ is even, let $X_n = A^{\infty} \cap \bigcup_{x\in X_{n-1}} \Gamma_x^u$. By construction, $X_1$ is $\frac{2}{10}$-dense in $(-1,1)^d$, meaning that 
\begin{equation*}
 \sup\{r>0: \exists x\in (-1,1)^d\text{ s.t. } B_r(x) \cap X_1 = \emptyset\} \leq \frac{2}{10}.
\end{equation*}
Similarly, $X_2$ is $\frac{2}{10^2}$-dense in $(-1,1)^d$ and, generally, $X_n$ is $\frac{2}{10^n}$-dense in $(-1,1)^d$. The proof is thus complete.

\end{proof}

We are now ready to proceed with the proof of the first auxiliary result.

\begin{proof}[Proof of Proposition \ref{bi-inflatable implies ergodicity}]

Take $\rho>0$ as in Lemma \ref{one large component implies local ergodicity}. We cover $M$ by a finite number of balls of radius $\rho/2$:

\[B_i = B_{\rho/2}(x_i); \quad i=1, \ldots, N.\]

We may suppose, without loss of generality, that $m(B_1 \cap A)>0$. Now, by connectedness of $M$, given any $k \in \{1, \ldots, N\}$, there exists a sequence $1=i_1, \ldots, i_{\ell} =k$ such that $B_{i_n} \cap B_{i_{n+1}} \neq \emptyset $ for $n = 1, \ldots, \ell-1$.

By applying Lemma \ref{one large component implies local ergodicity} to a typical point in $A \cap B_1$ we see that $m(B_{i_1} \cap A^c)=0$. Since $m(B_{i_1} \cap B_{i_2}) >0$ we can apply Lemma \ref{one large component implies local ergodicity} again to see that $m(B_{i_1} \cap A^c)= 0$. By repeating the argument we eventually conclude that $m(B_k \cap A^c) = 0$. Since $k \in \{1, \ldots, N\}$ was chosen arbitrarily, this means that $A$ has full Lebesgue measure in each $B_i$, hence it has full Lebesgue measure in $M$.
\end{proof}

The proof of the second auxiliary result relies on the following well known fact (see e.g. \cite{MR1944399}).

\begin{prop} \label{semicont}
 The integrated Lyapunov exponents
̈́\begin{align}
  \mathcal{DS}^1 & \rightarrow \mathbb{R} \\
f  \mapsto \int \lambda_-^{cu}(f, x) dm(x), & \quad 
f  \mapsto \int \lambda_+^{cs}(f, x) dm(x)
 \end{align}
are lower and upper semi-continuous, respectively.
\end{prop}

Recall the notation 
\begin{align}
\Xi_n^{cu} (f) & = \log \left( \sup_{x \in M} \| \wedge^{cu-1}Df_{\vert E_x^{cu}}^n \| \right) \\
\xi_n^{cu}(f, x) & = \log \vert \det Df_{\vert E_x^{cu} }^n \vert
\end{align}
from section \ref{small sight}. We define $\Xi_n^{cs}(f)$ and $\xi_n^{cs}$ similarly.

\begin{proof}[Proof of Propositon \ref{EBINA is open}]
 
Suppose $f \in \NA$ is ergodic and bi-inflatable. We must show that each $g \in \DS$, sufficiently close to $f$ in the $C^1$ topology, is also non-uniformly Anosov and ergodic. 

From Proposition \ref{semicont}, there is a $C^1$-neighbourhood $\mathcal{U}$ of $f$ such that for every $g\in \mathcal{U}$, we have 
\begin{equation}
 m(\{x : \lambda_+^{cs}(g,x) <0 \} \cap \{ x: \lambda_-^{cu}(g,x)>0\}) > 1/2.
\end{equation}
Furthermore, since $\frac{1}{n} \xi_n^{cu} (f,x)$, $\frac{1}{n} \xi_n^{cs} (f,x)$ are uniformly bounded sequences with 
\begin{align}
 \lim_{n\rightarrow \infty} \frac{1}{n} \xi_n^{cu}(f,x)  & = \int \xi_0^{cu} dm > \inf_{n\geq 0} \frac{1}{n} \Xi_n^{cu}(f) \\
\lim_{n\rightarrow \infty} \frac{1}{n} \xi_n^{cs}(f,x)  & = \int \xi_0^{cs} dm > \inf_{n\geq 0} \frac{1}{n} \Xi_n^{cs}(f),
\end{align}
there is some $N \geq 0$ such that, upon possibly reducing $\mathcal{U}$, we may suppose that for every $g \in \mathcal{U}$, we have
\begin{align}
& \inf_{m(A) > \frac{1}{2}} \frac{1}{m(A)} \int_A \xi_N^{cu}(g,x) dm(x) > \Xi_N^{cu} (g) \\
& \inf_{m(A) > \frac{1}{2}} \frac{1}{m(A)} \int_A \xi_N^{cs}(g,x) dm(x) > \Xi_N^{cs} (g).
\end{align}
It follows that every $g \in \mathcal{U}$ has at least one ergodic component $A$ which is bi-inflatable and on such that $\lambda_+^{cs}(g,x) <0$ and $\lambda_-^{cu}(g,x)>0$ for $m$-almost every $x \in A$. Hence it follows from Theorem \ref{main} and Proposition \ref{bi-inflatable implies ergodicity} that  $m(A) = 1$, proving that $f \in \mathcal{E} \cap \BI \cap  \NA$. 
\end{proof}

\bibliography{ergodicity}{}
\bibliographystyle{plain}

\end{document}